\documentclass{article}
\usepackage[utf8]{inputenc}
\usepackage{fullpage}

\usepackage{amsmath,amsthm,amssymb}
\usepackage{mathtools}
\usepackage{tikz}
\usepackage{hyperref,cleveref}
\usepackage{todonotes}

\newtheorem{theorem}{Theorem}[section]
\newtheorem{corollary}[theorem]{Corollary}
\newtheorem{lemma}[theorem]{Lemma}

\newtheorem{proposition}[theorem]{Proposition}
\newtheorem{definition}[theorem]{Definition}

\newtheorem{remark}[theorem]{Remark}

\AddToHook{env/theorem/begin}{\crefalias{section}{theorem}}
\AddToHook{env/corollary/begin}{\crefalias{theorem}{corollary}}
\AddToHook{env/lemma/begin}{\crefalias{theorem}{lemma}}
\AddToHook{env/observation/begin}{\crefalias{theorem}{observation}}
\AddToHook{env/claim/begin}{\crefalias{theorem}{claim}}
\AddToHook{env/proposition/begin}{\crefalias{theorem}{proposition}}
\AddToHook{env/definition/begin}{\crefalias{theorem}{definition}}
\AddToHook{env/fact/begin}{\crefalias{theorem}{fact}}
\AddToHook{env/remark/begin}{\crefalias{theorem}{remark}}

\usepackage{xspace}

\begin{document}

\title{Coloring t-perfect graphs with fewer colors}
\author{Matija Novakovi\'c \and Stefan Weltge}
\date{October 29, 2026}

\maketitle

\begin{abstract}
    Recently, Chudnovsky, Cook, Davies, Oum, and Tan obtained the first finite bound on the chromatic number of t-perfect graphs, showing that they are 199053-colorable.
    We improve this bound to 130 by refining their proof.

    The original proof establishes that every graph with large odd girth and large chromatic number contains a certain structure called an $r$-arithmetic rope, and that its existence in a certain leveling of a graph with large odd girth would imply an odd wheel as a t-minor, a known obstruction to t-perfectness.
    While their technique requires a lower bound on the chromatic number that is exponential in $r$, we show that the existence of an $r$-arithmetic rope can already be guaranteed under a linear bound.
    Using a slightly weakened notion of arithmetic ropes allows us to reduce the bound even further.
\end{abstract}

\section{Introduction}

The stable set problem asks for finding a maximum stable set in a given undirected graph, i.e., a largest set of pairwise non-adjacent vertices.
While the problem is NP-hard in general, it can be solved in polynomial time for several classes of graphs.
A common approach to derive efficient algorithms is to exploit the structure of the \emph{stable set polytope}, which is the convex hull of characteristic vectors of stable sets.
A prominent example, introduced by Chvátal~\cite{chvatal1975certain}, is the class of \emph{t-perfect graphs}, which are defined as the graphs $G$ whose stable set polytope coincides with the set of vectors $x \in [0,1]^{V(G)}$ satisfying $x_u + x_v \le 1$ for every edge $uv \in E(G)$ (edge inequalities) and $\sum_{v \in V(C)} x_v \le \frac12 (|V(C)| - 1)$ for every odd cycle $C$ of $G$ (odd cycle inequalities).

In 1992, Shepherd~\cite[8.14]{jensen2011graph} asked whether the stable set polytope $P$ of a t-perfect graph $G$ always has the integer decomposition property, that is, for every positive integer $k$, every integer point in $kP$ is a sum of $k$ integer vectors points in $P$.
Since the all-$\frac{1}{3}$ vector is contained in $P$, this would imply that $G$ is 3-colorable.
In 1994, Laurent and Seymour~\cite[p. 1207]{schrijver2003combinatorial} (and later Benchetrit~\cite{benchetrit2015proprietes,benchetrit20164}) found t-perfect graphs with chromatic number 4, negatively answering Shepherd's question.
Since then, it has been an open question whether every t-perfect graph is 4-colorable~\cite[8.14]{jensen2011graph}.

Only recently, Chudnovsky, Cook, Davies, Oum, and Tan~\cite{chudnovsky2024colouringtperfectgraphs} gave the first finite bound on the chromatic number of t-perfect graphs, showing that they are 199053-colorable.
While their proof has not been optimized for the best bound, they expected that carrying out their arguments with more care
\begin{quote}
    ``[...] would still leave a large gap between our bound and the lower bound of 4. A good milestone for narrowing the gap would be to improve our upper bound to at most 1000.''
\end{quote}
We show how their proof can be refined to obtain the following bound.

\begin{theorem}\label{thm:main}
    Every t-perfect graph is 130-colorable.
\end{theorem}

The approach of~\cite{chudnovsky2024colouringtperfectgraphs} is based on only two properties of t-perfect graphs: they do not contain odd wheels as t-minors, and they have a stable set that intersects every shortest odd cycle.
As the main technical contribution, the authors show that every graph with large odd girth and large chromatic number contains a certain structure called an ``arithmetic rope'' and observe that the existence of an arithmetic rope in a certain leveling of a graph with large odd girth would imply the existence of an odd wheel as a t-minor.
We refine their proof and show that the existence of an arithmetic rope can be guaranteed under a much weaker assumption on the chromatic number.
To decrease the bound on the chromatic number even further, we exploit the fact that a slightly weaker structure, which we call a \emph{quasi-arithmetic rope}, is already sufficient for finding an odd wheel t-minor.

\Cref{thm:main} also yields an improved bound on the chromatic number of h-perfect graphs, which are defined as the graphs $G$ whose stable set polytope coincides with the set of vectors $x \in [0,1]^{V(G)}$ satisfying the odd cycle inequalities and $\sum_{v \in Q} x_v \leq 1$ for every clique $Q$ of $G$ (clique inequalities).
Using a result of Sebő (unpublished, see~\cite{bruhn2012claw}) and Marcus~\cite{marcus1996shortestoddcycles}, the authors of~\cite[Thm. 1.2]{chudnovsky2024colouringtperfectgraphs} observe that if every t-perfect graph with $\omega(G) \geq 3$ is $k$-colorable, then every h-perfect graph $G$ is $(\omega(G) + k - 3)$-colorable.
Hence, we obtain the following bound.
\begin{corollary}
    Every h-perfect graph $G$ is $(\omega(G) + 127)$-colorable.
\end{corollary}

In Section~\ref{sec:basics}, we introduce definitions of the aforementioned structures, review the essential steps of the proof of~\cite{chudnovsky2024colouringtperfectgraphs}, describe our main result on the existence of quasi-arithmetic ropes, and explain how this leads to an improved bound on the chromatic number of t-perfect graphs.
The proof of our main result is given in Sections~\ref{sec:quasi_arithmetic_rope}, \ref{sec:segments_odd_and_even}, and \ref{sec:segments_even}.

\section{Quasi-arithmetic ropes and t-minors of odd wheels}
\label{sec:basics}

An essential property of t-perfect graphs is that they are closed under taking t-minors. 
This operation is based on the notion of a \emph{t-contraction}, which refers to the operation of contracting all edges in the closed neighborhood of a vertex whose neighborhood is a stable set.
A \emph{t-minor} of a graph $G$ is a graph obtained from $G$ by a sequence of vertex deletions and t-contractions.
Gerards and Shepherd~\cite{tminorstperfectgerards1998} observed that if $G$ is t-perfect, each of its t-minors is t-perfect as well.

A basic family of graphs that are known not to be t-perfect are the \emph{odd wheels}, which are obtained from odd cycles by adding a new central vertex that is adjacent to all vertices of the cycle.
Consequently, t-perfect graphs do not contain odd wheels as t-minors.

The main contribution of~\cite{chudnovsky2024colouringtperfectgraphs} is to show that graphs with sufficiently large chromatic number and odd girth must contain an odd wheel as a t-minor. Here, the odd girth of a graph is defined as the length of its shortest odd cycle. 
To bound the chromatic number of a t-perfect graph, it is indeed sufficient to consider graphs with large odd girth due to the following fact.

\begin{proposition}[{Marcus~\cite{marcus1996shortestoddcycles}, see also \cite[Lem.~3.4]{chudnovsky2024colouringtperfectgraphs}}]\label{prop:marcus}
    Every t-perfect graph $G$ has a stable set that intersects every shortest odd cycle in $G$.
\end{proposition}

To extract odd wheel t-minors, the authors of~\cite{chudnovsky2024colouringtperfectgraphs} introduce the notion of an arithmetic rope.
We consider the slightly different notion of a \emph{quasi-$r$-arithmetic rope} in a graph $G$, which consists of vertices $q_1,\dots,q_r$ together with paths $Q_{i,1}, Q_{i,2}$ for $i \in [r] = \{1, \dots, r\}$ such that (taking indices modulo $r$)
\begin{itemize}
    \item[a)] each $Q_{i,1}$ and $Q_{i, 2}$ is a $q_iq_{i + 1}$-path,
    \item[b)] each $Q_{i,1}$ has odd length,
    \item[c)] $Q_{i, 2}$ or $Q_{i + 1, 2}$ has even length,
    \item[d)] for every $(h_1,\dots,h_r) \in \{1,2\}^r$, the graph $Q_{1,h_1} \cup \dots \cup Q_{r,h_r}$ is an induced cycle, and
    \item[e)] $q_1, \dots, q_r$ have pairwise $G$-distance at least 6.
\end{itemize}
If all vertices of the paths $Q_{i,1}, Q_{i,2}$ are contained in a vertex set $X \subseteq V(G)$, we say that \emph{$G$ has a quasi-$r$-arithmetic rope in $X$}.
Note that, in this case, condition e) still refers to the distance in $G$ and not the distance in $G[X]$.
When constructing quasi-arithmetic ropes, we will make use of the fact that for some $i$, both $Q_{i,1}$ and $Q_{i,2}$ are allowed to have odd length.
For notational convenience, we will simply set $Q_{i, 2} = Q_{i, 1}$ in these cases.

Let $N_G^t(u)$ denote the set of vertices with distance exactly $t$ from $u$ in $G$ and let $N_G^t[u] = \bigcup_{i = 0}^t N_G^i(u)$ denote the corresponding closed neighborhood, where we omit the subscript when it is clear from the context.
A key observation in the proof of~\cite{chudnovsky2024colouringtperfectgraphs} is the following.
\begin{proposition}\label{prop:t_minor_in_rope}
    Let $G$ be a graph with odd girth at least 7, $u \in V(G)$, and $t \ge 4$. If $G$ contains a quasi-$5$-arithmetic rope in $N_G^t(u)$, then $G$ has an odd wheel t-minor.
\end{proposition}
Strictly speaking, the above statement is only shown for arithmetic ropes, which differ from our notion of quasi-arithmetic ropes in the sense that each $Q_{i,2}$ is required to have even length, and the pairwise $G$-distance between the $q_i$'s to be at least 5. However, their proof only relies on the fact that, for every choice of three vertices among $q_1,\dots,q_5$, there is a choice of paths along the $5$-arithmetic rope such that the resulting induced cycle is odd and partitioned into three odd-length paths by the three vertices. Since this is also true for our notion of quasi-$5$-arithmetic ropes, their proof directly carries over to our setting, so we refrain from repeating it.

Our main contribution is the following theorem.

\begin{theorem}\label{thm:quasi_rope_existence}
    For $r \in \mathbb{N}$, every graph $G$ with odd girth at least 15 and $X \subseteq V(G)$ contains a quasi-$r$-arithmetic rope in $X$ if 
    \[\chi(X) \geq \begin{cases}
        11r + 5 & \text{if } r \text{ is even,}\\
        11r + 8 & \text{if } r \text{ is odd.}
    \end{cases}\]
\end{theorem}

To see how the above result implies \Cref{thm:main}, assume that $G$ is a t-perfect graph with $\chi(G) \geq 131$. According to \Cref{prop:marcus}, we may iteratively remove six stable sets from $G$ to obtain an induced subgraph $H$ with odd girth at least 15. Note that $\chi(H) \geq 125$. Let $u$ be an arbitrary vertex of a connected component of $H$ with maximum chromatic number, and note that there must be a level $t$ with $\chi(N_H^t(u)) \geq \lceil \frac{125}{2} \rceil= 63 = 11\cdot 5 + 8$. Since the odd girth of $H$ is at least 15, $N_H^3[u]$ is bipartite and thus $t \geq 4$. By \Cref{thm:quasi_rope_existence}, $H$ contains a quasi-5-arithmetic rope in $N_H^t(u)$. \Cref{prop:t_minor_in_rope} now implies that $H$ contains an odd wheel t-minor, and so does $G$, contradicting the fact that $G$ is t-perfect.

The authors of~\cite[Thm. 4.1]{chudnovsky2024colouringtperfectgraphs} prove the following variant of the above result:
For $r \in \mathbb{N}$, every graph $G$ with odd girth at least 11 and $X \subseteq V(G)$ contains an $r$-arithmetic rope in $X$ if 
\[ \chi(X) \geq 6^{r + 1} + \frac{34}{5} (6^r - 1) - 1.\]
The exponential dependence on $r$ stems from constructing segments of the arithmetic rope by recursively restricting to levels of large chromatic number (causing a factor of 2 in each step), and dividing each such level into three parts (yielding an additional factor of 3).
By constructing segments in a more global fashion and avoiding the case distinction, we can show the existence of an $r$-arithmetic rope already for the much weaker bound $\chi(X) \geq 18r + 5$ under the same assumptions; we omit the details.
Since a quasi-arithmetic rope only requires some of the segments to contain an even-length path, the bound can be weakened further.
In case of a slightly larger odd girth, this bound can be improved even further, leading to \Cref{thm:quasi_rope_existence}.

The remainder of this paper is devoted to the proof of \Cref{thm:quasi_rope_existence}.

\begin{figure}[ht]
\centering
\begin{tikzpicture}[
    scale=0.72,
    vertex/.style={circle, fill=black, inner sep=1.4pt},
    lab/.style={font=\Large}
]




\coordinate (a1) at (-3.2,  1.1);
\coordinate (a2) at (-0.3,  1.2);
\coordinate (a3) at ( 2.7,  0.0);
\coordinate (a4) at ( 1.2, -2.1);
\coordinate (a5) at (-2.4, -2.1);

\draw[dashed, thick]
  (a1) .. controls (-2.4,1.1) and (-1.7,2.2) .. (a2);

\draw[dotted, thick]
  (a1) .. controls (-2.1,1.5) and (-1.3,0.7) .. (a2);

\draw[dashed, thick]
  (a2) .. controls (0.3,0.3) and (1.6,0.7) .. (a3);

\draw[dotted, thick]
  (a2) .. controls (0.8,2.2) and (2.0,1.4) .. (a3);

\draw[dashed, thick]
  (a3) .. controls (2.4,-0.8) and (3.3,-2.4) .. (a4);

\draw[dotted, thick]
  (a3) .. controls (3.2,-0.9) and (2.6,-1.9) .. (a4);

\draw[dashed, thick]
  (a4) .. controls (0.3,-1.8) and (-0.6,-1.5) .. (a5);

\draw[dotted, thick]
  (a4) .. controls (0.0,-2.6) and (-0.8,-2.5) .. (a5);

\draw[dashed, thick]
  (a5) .. controls (-3.4,-1.8) and (-3.6,-0.5) .. (a1);

\draw[dotted, thick]
  (a5) .. controls (-2.7,-0.8) and (-2.5,0.6) .. (a1);

\node[vertex] at (a1) {};
\node[vertex] at (a2) {};
\node[vertex] at (a3) {};
\node[vertex] at (a4) {};
\node[vertex] at (a5) {};

\node[lab] at (-3.55,  1.45) {$q_1$};
\node[lab] at (-0.30,  1.85) {$q_2$};
\node[lab] at ( 3.05,  0.35) {$q_3$};
\node[lab] at ( 1.00, -1.60) {$q_4$};
\node[lab] at (-2.05, -1.65) {$q_5$};





\begin{scope}[xshift=2cm]
\coordinate (b1) at (4.0,  1.1);
\coordinate (b2) at (6.9,  1.2);
\coordinate (b3) at (9.9,  0.0);
\coordinate (b4) at (8.4, -2.1);
\coordinate (b5) at (4.8, -2.1);

\draw[dashed, thick]
  (b1) .. controls (4.8,1.1) and (5.5,2.2) .. (b2);

\draw[dotted, thick]
  (b1) .. controls (5.1,1.5) and (5.9,0.7) .. (b2);

\draw[dotted, thick]
  (b2) .. controls (8.0,2.2) and (9.2,1.4) .. (b3);

\draw[dashed, thick]
  (b3) .. controls (9.6,-0.8) and (10.5,-2.4) .. (b4);

\draw[dotted, thick]
  (b3) .. controls (10.4,-0.9) and (9.8,-1.9) .. (b4);

\draw[dotted, thick]
  (b4) .. controls (7.2,-2.6) and (6.4,-2.5) .. (b5);

\draw[dashed, thick]
  (b5) .. controls (3.8,-1.8) and (3.6,-0.5) .. (b1);

\draw[dotted, thick]
  (b5) .. controls (4.5,-0.8) and (4.7,0.6) .. (b1);

\node[vertex] at (b1) {};
\node[vertex] at (b2) {};
\node[vertex] at (b3) {};
\node[vertex] at (b4) {};
\node[vertex] at (b5) {};

\node[lab] at (3.65,  1.45) {$q_1$};
\node[lab] at (6.90,  1.85) {$q_2$};
\node[lab] at (10.25, 0.35) {$q_3$};
\node[lab] at (8.20, -1.60) {$q_4$};
\node[lab] at (5.15, -1.65) {$q_5$};
\end{scope}

\end{tikzpicture}
\caption{The left figure shows a $5$-arithmetic rope, while the right figure shows a quasi-$5$-arithmetic rope. Dotted lines represent odd-length paths, and dashed lines represent even-length paths. Note that the pairwise $G$-distance between the $q_i$'s is at least $5$ for arithmetic ropes and at least $6$ for quasi-arithmetic ropes.}
\label{fig:arithmetic_comparison}
\end{figure}

\section{Constructing a quasi-arithmetic rope}
\label{sec:quasi_arithmetic_rope}

In what follows, we say that a vertex set $B$ covers $C$ if $B$ and $C$ are disjoint and every vertex in $C$ has a neighbor in $B$. Moreover, we say that a vertex set $X$ is anticomplete to a vertex set $Y$ if every vertex in $X$ has no neighbor in $Y$.

In analogy to~\cite{chudnovsky2024colouringtperfectgraphs}, let us consider the notion of a \emph{broken quasi-$r$-arithmetic rope}, which consists of vertices $q_1, \dots, q_{r + 1}$ together with paths $Q_{i, 1}, Q_{i, 2}$ for $i \in [r]$ such that 

\begin{itemize}
    \item each $Q_{i, 1}$ and $Q_{i, 2}$ is a $q_iq_{i+1}$-path,
    \item each $Q_{i, 1}$ has odd length,
    \item for every $h \in \{1, 2\}^r$, $Q_{1, h_1} \cup \dots \cup Q_{r, h_r}$ is an induced path, and
    \item the vertices $q_1, \dots, q_{r + 1}$ are pairwise at $G$-distance at least 6.
\end{itemize}

We call a broken quasi-$r$-arithmetic rope \emph{alternating} if $Q_{i, 2}$ has even length whenever $i \in [r]$ is even, and call it \emph{good} if, in addition, $Q_{r,2}$ has even length.
Clearly, the two notions are equal if $r$ is even.
We will prove the following.

\begin{lemma}\label{lemma:construction_almost_broken}
    Let $c, r \geq 1$, let $G$ be a graph with odd girth at least 15, and let $B, C \subseteq V(G)$ be disjoint with $B$ covering $C$. Let $q_1 \in V(G) \setminus C$ be a vertex such that $G[C \cup \{q_1\}]$ is connected and $\chi(C) \geq 7 \left\lfloor \frac{r}{2} \right\rfloor + 4 \left\lceil \frac{r}{2} \right\rceil + c$. Then there exist $B' \subseteq B, C' \subseteq C, q_2, q_3, \dots, q_{r + 1} \in C \setminus C'$ and an alternating broken quasi-$r$-arithmetic rope $R$ on vertices $q_1,\dots,q_{r+1}$ with paths $(Q_{i, 1}, Q_{i, 2})_{i = 1}^r$ such that 
    \begin{itemize}
        \item $G[C' \cup \{q_{r + 1}\}]$ is connected,
        \item $\chi(C') \geq c$,
        \item $B'$ covers $C'$,
        \item $B'$ is anticomplete to $(V(Q_{i, 1}) \cup V(Q_{i, 2})) \setminus N^3[q_{i + 1}]$ for   all $i \in [r]$,
        \item $V(Q_{i, 1}) \cup V(Q_{i, 2}) \subseteq C \cup \{q_1\} \cup \left(B \cap N^2[q_{i + 1}] \right)$ for all $i \in [r]$,
        \item $C'$ is anticomplete to $\left( \bigcup_{i = 1}^r (V(Q_{i, 1}) \cup V(Q_{i, 2})) \right) \setminus \{q_{r + 1}\}$, and
        \item the $G$-distance between $\{q_1, \dots, q_{r}\}$ and $C' \cup \{q_{r+  1}\}$ is at least $6$. 
    \end{itemize}
    If $\chi(C) \geq 7 \left\lceil \frac{r}{2} \right\rceil + 4 \left\lfloor \frac{r}{2} \right\rfloor + c$, then $R$ can even be chosen to be good.
\end{lemma}

Let us first explain how the above statement yields our main claim.

\begin{proof}[Proof of \Cref{thm:quasi_rope_existence}]
    Let $G$ be a graph with odd girth at least 15 and let $X \subseteq V(G)$ with 
    \[\chi(X) \geq \begin{cases}
        11r + 5 & \text{if } r \text{ is even,}\\
        11r + 8 & \text{if } r \text{ is odd.}
    \end{cases}\]
Without loss of generality, we may assume that $G[X]$ is connected, as otherwise we can choose a connected component of $G[X]$ with maximum chromatic number. 
Let $v$ be some vertex of $X$ and let $L_i = N^i_X(v)$ for all $i \geq 0$. Since $G$ has odd girth at least 15, there is an integer $s \geq 6$ such that $\chi(L_{s + 1}) \geq \lceil \chi(X) /2 \rceil$.
Hence if $r$ is even, we have
\[
\chi(L_{s + 1}) \geq \left\lceil \frac{\chi(X)}{2} \right \rceil \geq 5.5r + 3 = 7 \left \lceil \frac{r}{2} \right \rceil + 4 \left \lfloor \frac{r}{2} \right \rfloor + 3,
\]
and if $r$ is odd, we have
\[
\chi(L_{s + 1}) \geq \left\lceil \frac{\chi(X)}{2} \right \rceil \geq 5.5r + 4.5 = 7 \left \lceil \frac{r}{2} \right \rceil + 4 \left \lfloor \frac{r}{2} \right \rfloor + 3.
\]

Let $C$ be the set of vertices of a connected component of $G[L_{s + 1}]$ with maximum chromatic number. Let $q_1$ denote a vertex of $L_s$ that is adjacent to some vertex of $C$ and let $B = L_s$. \Cref{lemma:construction_almost_broken} implies that there exist $B' \subseteq B, C' \subseteq C$, and a good broken quasi-$r$-arithmetic rope consisting of ends $q_1$ and $q_2, \dots, q_{r + 1} \in C \setminus C'$ and paths $Q_{i, 1}$ and $Q_{i, 2}$ for each $i \in [r]$ such that 
\begin{itemize}
    \item $G[C' \cup \{q_{r+ 1} \}]$ is connected,
    \item $\chi(C') \geq 3$,
    \item $B'$ covers $C'$,
    \item $B'$ is anticomplete to $(V(Q_{i, 1}) \cup V(Q_{i, 2})) \setminus N^3[q_{i + 1}]$ for   all $i \in [r]$,
    \item $V(Q_{i, 1}) \cup V(Q_{i, 2})$ is contained in $C \cup \{q_1\} \cup (B \cap N^2[q_{i + 1}])$ for all $i \in [r]$,
    \item $C'$ is anticomplete to $\left( \bigcup_{i = 1}^r (V(Q_{i, 1}) \cup V(Q_{i, 2})) \right) \setminus \{q_{r + 1}\}$, and
    \item the $G$-distance between $\{q_1, \dots, q_{r}\}$ and $C' \cup \{q_{ r+  1}\}$ is at least $6$. 
\end{itemize}
Since $\chi(C') \geq 3$ and $\chi(N^5_G[q_{r + 1}]) \leq 2$, there exists a vertex $x \in C'$ with $G$-distance at least $6$ from $q_{r + 1}$. Since $x \in C'$, it also has $G$-distance at least $6$ from $q_1, \dots, q_r$. As $B'$ covers $C'$, there is a vertex $b \in B'$ adjacent to $x$. 

Let $P_1$ be a shortest path between $q_{r + 1}$ and $b$ in $G[\{q_{r + 1}, b \} \cup C']$. Observe that this is an induced path.
Now let $a_1 \in L_{s-1}$ be a vertex adjacent to $q_1$ and let $a_2 \in L_{s - 1}$ be a vertex adjacent to $b$.
Let $P_2$ be an induced path between $a_1$ and $a_2$ in $G[ \{a_1, a_2\} \cup \bigcup_{i = 0}^{s - 2} L_i]$. We claim that $ba_2P_2a_1q_1$ is again an induced path. It suffices to show that $\{b , a_2\}$ is anticomplete to $\{a_1, q_1\}$ since $V(P_2) \setminus \{a_1, a_2\} \subseteq \bigcup_{i = 0}^{s - 2} L_i$. This follows because the $G$-distance between vertices in $\{b, a_2\}$ and $x$ is at most $2$, but the $G$-distance between vertices in $\{a_1, q_1\}$ and $x$ is at least $5$.

Now let $P$ be the path $P_1ba_2P_2a_1q_1$ between $q_{r + 1}$ and $q_1$. To prove that it is induced, it suffices to show that $V(P_1) \setminus \{b\}$ is anticomplete to $\{q_1\}$ since $V(P_1) \setminus \{b\} \subseteq C \subseteq L_{s + 1}$ and $V(P_2) \subseteq \bigcup_{i = 0}^{s - 1} L_{i}$. This holds because $V(P_1) \setminus \{b\} \subseteq C' \cup \{q_{r + 1}\}$ and the $G$-distance between $q_1$ and $C' \cup \{q_{r + 1}\}$ is at least 6.

Finally, we show that $V(P) \setminus \{q_1, q_{r + 1}\}$ is anticomplete to $\bigcup_{i = 1}^r (V(Q_{i, 1}) \cup V(Q_{i, 2})) \setminus \{q_1, q_{r + 1}\}$. For $V(P_1) \setminus \{q_{r + 1}, b\} \subseteq C'$, the claim follows directly from \Cref{lemma:construction_almost_broken}. Since $V(P_2) \setminus \{a_1, a_2\} \subseteq \bigcup_{i = 0}^{s - 2} L_i$ and $\bigcup_{i = 1}^r (V(Q_{i, 1}) \cup V(Q_{i, 2})) \subseteq L_s \cup L_{s + 1}$, we only need to show that $\{a_1, a_2, b\}$ is anticomplete to $\bigcup_{i = 1}^r (V(Q_{i, 1}) \cup V(Q_{i, 2})) \setminus \{q_1, q_{r + 1}\}$. 
Since $a_1, a_2 \in L_{s - 1}$ and $(V(Q_{i, 1}) \cup V(Q_{i, 2})) \setminus\{q_1\} \subseteq C \cup (B \cap N^2[q_{i + 1}])$ for all $i \in [r]$, $a_1$ and $a_2$ may only have a neighbor in $B \cap N^2[q_{i + 1}]$ along $Q_{i, 1} \cup Q_{i, 2}$. Since $a_1$ and $a_2$ each have $G$-distance at least $4$ to $\{q_2, \dots, q_{r + 1}\}$, the respective anticompleteness condition follows.
Since $b \in B'$ and $B'$ is anticomplete to $(V(Q_{i, 1}) \cup V(Q_{i, 2})) \setminus N^3[q_{i + 1}]$ for all $i \in [r]$, the respective anticompleteness condition follows because $b$ has $G$-distance at least $5$ to $\{q_1, \dots, q_{r + 1}\}$.

Hence we obtain a quasi-$r$-arithmetic rope from the good broken quasi-$r$-arithmetic rope by extending $Q_{r, 1}$ and $Q_{r,2}$ with $P$. Note that we might need to relabel $Q_{r, 1}$ and $Q_{r, 2}$ depending on the parity of $P$, which is possible because the broken quasi-$r$-arithmetic rope is good.
\end{proof}

The broken quasi-arithmetic ropes in \Cref{lemma:construction_almost_broken} are constructed inductively.
Each step is based on the following result, which can be seen as a refinement of \cite[Lem. 4.4]{chudnovsky2024colouringtperfectgraphs}.

\begin{lemma}\label{lemma:contruction_two_paths_optimal}
    Let $c \geq 1$ and let $G$ be a graph with odd girth at least $13$, let $B, C \subseteq V(G)$ be disjoint vertex subsets with $B$ covering $C$, let $q \in V(G) \setminus C$ be a vertex such that $G[C \cup \{q\}]$ is connected and $\chi(C)\geq c+7$.

    Then there exist $B' \subseteq B, C' \subseteq C, q' \in C \setminus C'$ and two induced paths $Q_1, Q_2$ between $q$ and $q'$ in ${G[(C \setminus C') \cup (B \setminus B') \cup \{q\}]}$ such that 

    \begin{itemize}
        \item $G[C' \cup \{q'\}]$ is connected,
        \item $\chi(C') \geq c$, 
        \item $B'$ covers $C'$,
        \item $B'$ is anticomplete to $(V(Q_1) \cup V(Q_2)) \setminus N^3[q']$,
        \item $V(Q_1) \cup V(Q_2)$ is contained in $C \cup \{q\} \cup ((B \cap N^2[q']) \setminus N^4[q])$,
        \item $C'$ is anticomplete to $(V(Q_1) \cup V(Q_2)) \setminus \{q'\}$,
        \item the $G$-distance between $q$ and $C' \cup \{q'\}$ is at least $6$, and
        \item $Q_1$ has odd length and $Q_2$ has even length.
    \end{itemize}
\end{lemma}

For segments of the broken quasi-arithmetic rope that do not require an even-length path, we make use of the following variant that requires a weaker bound on the chromatic number at the cost of a slightly higher odd girth.

\begin{lemma}\label{lemma:construction_quasi}
    Let $c \geq 1$ and let $G$ be a graph with odd girth at least $15$, let $B, C \subseteq V(G)$ be disjoint vertex subsets with $B$ covering $C$, let $q \in V(G) \setminus C$ be a vertex such that $G[C \cup \{q\}]$ is connected and $\chi(C)\geq c+4$.

    Then there exist $B' \subseteq B, C' \subseteq C, q' \in C \setminus C'$ and an induced odd-length path $Q_1$ between $q$ and $q'$ in $G[(C \setminus C') \cup (B \setminus B') \cup \{q\}]$ such that 

    \begin{itemize}
        \item $G[C' \cup \{q'\}]$ is connected,
        \item $\chi(C') \geq c$, 
        \item $B'$ covers $C'$,
        \item $B'$ is anticomplete to $V(Q_1) \setminus N^3[q']$,
        \item $V(Q_1)$ is contained in $C \cup \{q\} \cup ((B \cap N^2[q']) \setminus N^4[q])$,
        \item $C'$ is anticomplete to $V(Q_1) \setminus \{q'\}$, and
        \item the $G$-distance between $q$ and $C' \cup \{q'\}$ is at least $6$.
    \end{itemize}
\end{lemma}

We close this section by showing how \Cref{lemma:construction_almost_broken} is obtained from the above two lemmas.
Their proofs will be discussed in the next two sections.

\begin{proof}[Proof of \Cref{lemma:construction_almost_broken}]

    We prove the first claim by induction and explain how it implies the second.
    The base case $r = 1$ follows from \Cref{lemma:construction_quasi}. Now let $r \geq 2$ and assume that the claim holds up to $r - 1$, i.e., there exist $B_{r-1} \subseteq B, C_{r - 1} \subseteq C, q_2, q_3, \dots, q_r \in C \setminus C_{r- 1}$, and an alternating broken quasi-$(r-1)$-arithmetic rope $(q_i, Q_{i, 1}, Q_{i, 2})_{i = 1}^{r - 1}$ satisfying the above properties with $c' =  4 + c$ if $r$ is odd or with $c' = 7 + c$ if $r$ is even.
    
    Depending on the parity of $r$, \Cref{lemma:contruction_two_paths_optimal} or \Cref{lemma:construction_quasi} with $B_{r-1}, C_{r-1}$, and $q_r$ implies that there are $B' \coloneqq B_r, C' \coloneqq C_r \subseteq C_{r - 1}, q_{r + 1} \in C_{r - 1} \setminus C_r$, and two induced paths $Q_{r, 1}, Q_{r,2}$ between $q_r$ and $q_{r + 1}$ in $G\left[ (C_{r - 1} \setminus C_r ) \cup (B_{r - 1} \setminus B_r) \cup \{q_{r}\} \right]$ satisfying the properties of the respective lemma. (Note that $Q_{r, 2}$ does not have to be an even-length path if $r$ is odd.) 
    It is easy to check that the properties listed above are satisfied. We only need to verify that $(q_i, Q_{i, 1}, Q_{i, 2})_{i = 1}^r$ is in fact an alternating broken quasi-$r$-arithmetic rope. It remains to check that the induced path condition holds since the rest is immediate.
    
    To this end, suppose for contradiction that there is some $(h_1, \dots, h_r)$ such that $Q_{1, h_1} \cup Q_{2, h_2 } \cup \dots \cup Q_{r, h_r}$ is not an induced path.
    By the induction hypothesis, $Q_{1, h_1} \cup Q_{2, h_2 } \cup\dots \cup Q_{r - 1, h_{r - 1}}$ is an induced path. Hence there is an edge between two vertices $x \in V(Q_{i, h_{i}}) \setminus \{q_r\}$ and $y \in V(Q_{r, h_r}) \setminus \{q_r\}$ for some $i \in [r - 1]$. Since $C_{r - 1}$ is anticomplete to $(V(Q_{i, 1}) \cup V(Q_{i, 2})) \setminus \{q_r\}$, we have that $y \in (B_{r - 1} \cap N^2[q_{r + 1}]) \setminus N^4[q_r]$. Since $B_{r - 1}$ is anticomplete to $(V(Q_{i, 1}) \cup V(Q_{i, 2})) \setminus N^3[q_{i + 1}]$, we have that $x \in N^3[q_{i + 1}]$. Hence there is a path of length at most 4 between $y$ and $q_{i + 1}$, which implies that $y \neq q_{r + 1}$. 
    
    Now suppose that $y$ does not have a neighbor in $C_{r - 1} \cap V(Q_{r,h_r})$. Note that this implies that $y \in N^2(q_{r+1})$ since $q_{r + 1} \in C_{r - 1}$. Since $V(Q_{r, h_r}) \subseteq C_{r - 1} \cup \{q_r\} \cup (B \cap N^2[q_{r + 1}])$, $y$ has two neighbors in $V(Q_{r, h_r}) \cap (B \cap N^2[q_{r + 1}])$, which we denote by $z_1$ and $z_2$. At most one of these can be adjacent to $q_{r +1 }$ since $Q_{r, h_r}$ is an induced path. Without loss of generality, we may assume that $z_1 \in N^2(q_{r + 1})$. Since $G$ is triangle-free, there is a vertex $b_{y}$ adjacent to $y$ and $q_{r + 1}$ and a vertex $b_{z_1}$ adjacent to $z_1$ and $q_{r + 1}$. However, together with $q_{r + 1}$ these vertices form a cycle of length 5, which contradicts that the odd girth of $G$ is at least $7$. Thus $y$ has a neighbor in $C_{r - 1} \cap V(Q_{r, h_r})$.
    
    Therefore if $i < r - 1$, there is a path of length at most 5 from $q_{i + 1}$ to a vertex in $C_{r - 1}$, which contradicts the induction hypothesis that the $G$-distance between $q_{i + 1}$ and $C_{r - 1}$ is at least 6. If $i = r - 1$, the $G$-distance between $q_r$ and $y$ is at most 4, which is again a contradiction.

    The second claim for good broken quasi-$r$-arithmetic ropes is equivalent to the first if $r$ is even. For odd $r$, the second claim follows from the first for $r-1$ and \Cref{lemma:contruction_two_paths_optimal} by arguing exactly as above.
\end{proof}

\section{Constructing segments with odd and even paths}
\label{sec:segments_odd_and_even}

We first prove the following lemma, which will be used in the proofs of \Cref{lemma:contruction_two_paths_optimal,lemma:construction_quasi}.
\begin{lemma}\label{lemma:exclude_bipartite}
    Let $k \geq 1$ and let $G$ be a graph with odd girth at least $2k + 3$. Let $X \subseteq V(G)$ with $\chi(X) \geq 3$, let $v \in V(G)$, $1 \leq \ell \leq k$, and let $B = N_G^{\ell}[v]$.
    Then 
    \[\chi(X \setminus B) \geq \chi(X) - 1.\]
\end{lemma}

\begin{proof}
	It suffices to prove the claim for $\ell = k$.
	For each $0 \leq i \leq k$, write $L_i = N_G^i(v)$. Since $G$ has odd girth at least $2k +3$, each $L_i$ is a stable set, as otherwise $G$ would contain an odd cycle of length $2i + 1 < 2k + 3$.
	Since any vertex in $L_i$ has neighbors only in $L_{i - 1}$ or $L_{i + 1}$, $B$ has a bipartition $(A_1, A_2)$ given by the union of all odd and even $L_i$, respectively. Let $A_2$ denote the part that contains $L_k$. Then the $G$-distance between $v$ and $X \setminus B$ is at least $\ell + 1$ and the $G$-distance between $v$ and $A_1$ is at most $\ell - 1$, so $X \setminus B$ is anticomplete to $A_1$.
	Hence, using $\chi(X \setminus B) \geq \chi(X) - \chi(B) \geq 1$, we have
	\[
		\chi(X \setminus A_2) = \chi \left( (X \setminus B) \cup (X \cap A_1) \right) = \chi(X \setminus B).
	\]
	Finally, since $A_2$ is a stable set,
	\[ 
	\chi(X \setminus B) = \chi(X \setminus A_2) \geq \chi(X) - \chi(A_2) \geq \chi(X) - 1. 
	\]
\end{proof}

\begin{definition}
    A stable grading is a sequence $(W_1, \dots, W_n)$ of pairwise disjoint stable sets of $V(G)$ such that their union is $V(G)$.
    
    A vertex $u \in V(G)$ is earlier than $v \in V(G)$ with respect to a stable grading $(W_1, \dots, W_n)$ if $u \in W_i$ and $v \in W_j$ for some $i < j$.
\end{definition}

We say that a vertex $w \in V(G)$ is left-active with respect to a vertex set $X$ if there exists an edge $uv$ of $G[X]$ such that $u$ and $v$ are earlier than $w$ and $w$ is adjacent to at least one of $u$ and $v$.

\begin{lemma}\label{lemma1:beyond_leveling}
    Let $c \geq 1$ and let $G$ be a graph with $\chi(G) \geq c + 5$ and $V(G) = C_0 \cup C_1 \cup C_2$.
    Furthermore, let \((C_0^t)_{t\geq 1}\) be a partition of \(C_0\) such that the ends of every edge of \(G[C_0]\) lie in parts whose indices differ by at most one.

    If $(W_1, \dots, W_n)$ is a stable grading of $V(G)$, there exist a subset $X$ of $V(G)$ and an edge $uv$ belonging to one of $G[C_1], G[C_2]$, or $G[C_0^t]$, for some $t \geq 1$, such that
    \begin{itemize}
        \item $G[X]$ is connected,
        \item $\chi(X) \geq c$,
        \item $u$ and $v$ are earlier than every vertex in $X$, and
        \item at least one of $u$ and $v$ has a neighbor in $X$.
    \end{itemize}
        \end{lemma}

\begin{proof}
    Let $\mathcal{D} = \{ C_1, C_2 \} \cup \{C_0^t : t \geq 1\}$. 
    Let $A$ be the set of vertices in $G$ that are left-active with respect to some $D \in \mathcal{D}$ and let $B = V(G) \setminus A$.
    
    We first show that $B$ is $5$-colorable.
    Define $T$ to be the set of vertices $v \in B$ such that there is a $D \in \mathcal{D}$ and $u \in D$ such that $uv$ is an edge of $G[D]$ and $u$ is earlier than $v$. 
    No vertex in $T$ can have a later neighbor in $T$, as otherwise the later vertex would be left-active with respect to some $D \in \mathcal{D}$, contradicting $T \subseteq B$.
    Since one vertex is earlier than the other in each pair of adjacent vertices, $T$ is a stable set.
	Moreover, $(D \cap B) \setminus T$ is a stable set for each $D \in \mathcal{D}$. To see this, suppose that $uv$ is an edge of $G[(D\cap B) \setminus T]$.
	Either $u$ or $v$ is earlier than the other vertex, so at least one of them belongs to $T$, a contradiction.
    For
    \[
    C_0^{\text{even}} = \bigcup_{t \geq 1 : t \text{ even}} C_0^t
    \quad \text{and}
    \quad
    C_0^{\text{odd}} = \bigcup_{t \geq 1 : t \text{ odd}} C_0^t,
    \]
    this implies that $(B \cap C_0^{\text{even}}) \setminus T$ and $(B \cap C_0^{\text{odd}}) \setminus T$ are stable sets since any edge of $G[C_0]$ is contained in $C_0^t \cup C_0^{t + 1}$ for some $t \geq 1$.
    Since $B$ is the union of the five stable sets $(B \cap C_0^{\text{odd}}) \setminus T, (B \cap C_0^{\text{even}}) \setminus T, (B \cap C_1) \setminus T, (B \cap C_2) \setminus T$, and $T$, we have $\chi(B) \leq 5$. Hence 
    \begin{align*}
        \chi(A) \geq \chi(G) - \chi(B) \geq c + 5 - 5 = c.
    \end{align*}
    
    Now, let $H$ be a connected component of $G[A]$ with maximum chromatic number, where we note that $\chi(H) = \chi(A) \geq c$. There is a minimum integer $i$ such that $1 \leq i \leq n$ and $W_i \cap V(H)$ is non-empty. Let $w$ be a vertex in $W_i \cap V(H)$. Since $w$ is left-active, there is an edge $uv$ of $G[C_1], G[C_2]$, or $G[C_0^t]$, for some $t \geq 1$, such that $u$ and $v$ are earlier than $w$ and $w$ is adjacent to one of $u$ and $v$. Since $i$ was chosen to be minimal, $u$ and $v$ are earlier than every vertex in $H$. Thus $X = V(H)$ satisfies the required conditions.
\end{proof}

In the triangle-free case, we can prove the following lemma. Its proof is identical to the proof of~\cite[Lem. 4.3]{chudnovsky2024colouringtperfectgraphs}. We include it for the sake of completeness.

\begin{lemma}\label{lemma2:beyond_leveling}
    Let $c \geq 1$ and let $G$ be a triangle-free graph with $\chi(G) \geq c + 6$ and $V(G) = C_0 \cup C_1 \cup C_2$. Furthermore, let \((C_0^t)_{t\geq 1}\) be a partition of \(C_0\) such that the ends of every edge of \(G[C_0]\) lie in parts whose indices differ by at most one.

    If $(W_1, \dots, W_n)$ is a stable grading of $V(G)$, there exist a subset $X$ of $V(G)$ and an edge $uv$ belonging to one of $G[C_1], G[C_2]$, or $G[C_0^t]$, for some $t \geq 1$, such that
    \begin{itemize}
        \item $G[X]$ is connected,
        \item $\chi(X) \geq c$,
        \item $u$ and $v$ are earlier than every vertex in $X$, 
        \item $u$ has no neighbor in $X$, and
        \item $v$ has a neighbor in $X$.
    \end{itemize}
    \end{lemma}

\begin{proof}
By \Cref{lemma1:beyond_leveling}, there exist a subset $X'$ of $V(G)$ and an edge $u'v'$ of $G[C_1], G[C_2]$, or $G[C_0^t]$ for some $t \geq 1$ with the following properties:  
\begin{itemize}
    \item $G[X']$ is connected,
    \item $\chi(X') \geq c + 1$,
    \item $u'$ and $v'$ are earlier than every vertex in $X'$, 
    \item $v'$ has a neighbor in $X'$.
\end{itemize}
Let $D$ be a connected component of $G[X' \setminus N(v')]$ with maximum chromatic number. Since $G$ is triangle-free, we have that 
\[\chi(D) \geq \chi(X') - \chi(N(v')) \geq \chi(X') - 1 \geq c.\]
If $u'$ has a neighbor in $V(D)$, then the lemma follows by setting $u = v', v = u'$, and $X = V(D)$. Hence we may assume that $u'$ has no neighbor in $V(D)$. Let $w$ be a neighbor of $v'$ in $X'$ that has a neighbor in $V(D)$. Note that $w$ is non-adjacent to $u'$ since $G$ is triangle-free. The lemma follows by setting $u = u', v = v'$, and $X = V(D) \cup \{w\}$.
\end{proof}

We are now ready to prove \Cref{lemma:contruction_two_paths_optimal}.

\begin{proof}[Proof of \Cref{lemma:contruction_two_paths_optimal}]
For $i \geq 0$, let $M_i$ denote the set of vertices in $C \cup \{q\}$ with $G[C \cup \{q\}]$-distance exactly $i$ from $q$. Let $d$ be the unique integer such that $M_{d} \neq \emptyset$ and $M_{d'} = \emptyset$ for all $d' > d$.

\textbf{Finding structure in $B$ and $C$: }
For $1 \leq t \leq d$, let $\mathcal{M}^{\leq t}  = M_0 \cup M_1 \cup \dots \cup M_t$ and $\mathcal{M}^{\geq t} = M_{t} \cup M_{t + 1} \cup \dots \cup M_{d}$.
For \(6 \leq t \leq d\), partition \(B\) into three sets \(B_0^t, B_1^t, B_2^t\) as follows. Let \(B_0^t\) be the set of vertices in \(B\) with no neighbor in \(\mathcal M^{\le t-3}\). Let \(B_1^t\) be the set of vertices \(v \in B\) with a neighbor in \(\mathcal M^{\le t-3}\) for which the \(G[\mathcal M^{\le t-3} \cup \{v\}]\)-distance between \(q\) and \(v\) is odd, and let \(B_2^t\) be the set of vertices for which this distance is even.

For each $6 \le t \le d$, let $C_0^t$ be the set of vertices in
$M_t \setminus N_G^5[q]$ that have a neighbor in $B_0^t$ but no neighbor in $B_1^t \cup B_2^t$. For every $v \in M_t \setminus N_G^5[q]$ with a neighbor in $B_1^t \cup B_2^t$, choose a neighbor $b_v \in B_1^t \cup B_2^t$ such that the $G[\mathcal{M}^{\leq t - 3} \cup \{b_v\}]$-distance between $q$ and $b_v$ is minimal. Then put $v$ in $C_h^t$ such that $b_v \in B_h^t$. This assignment is unique and well-defined since $B_1^t$ and $B_2^t$ are disjoint and the distance above has different parity for the two sets. In particular, it is independent of the choice of $b_v$. 
Finally, set $C_i = \bigcup_{t=6}^{d} C_i^t$ for each $i \in \{0, 1, 2\}$.
Since $B$ covers $C$, it follows that $C \setminus N_G^5[q] = C_0 \cup C_1 \cup C_2$, and that $B_i^t$ covers $C_i^t$ for every $6 \le t \le d$ and $i \in \{0, 1, 2\}$.
Note that we have that $B_0^t \supseteq B_0^{t + 1}$ and $B_h^t \subseteq B_h^{t + 1}$ for all $h \in \{1, 2\}$ and $6 \leq t \leq d - 1$.

\textbf{Constructing a suitable stable grading: }
The constructed stable grading is illustrated in \Cref{fig:grading}.
\begin{figure}[ht]
    \centering
    \includegraphics[width=\textwidth, trim={-71 0 71 0}, clip]{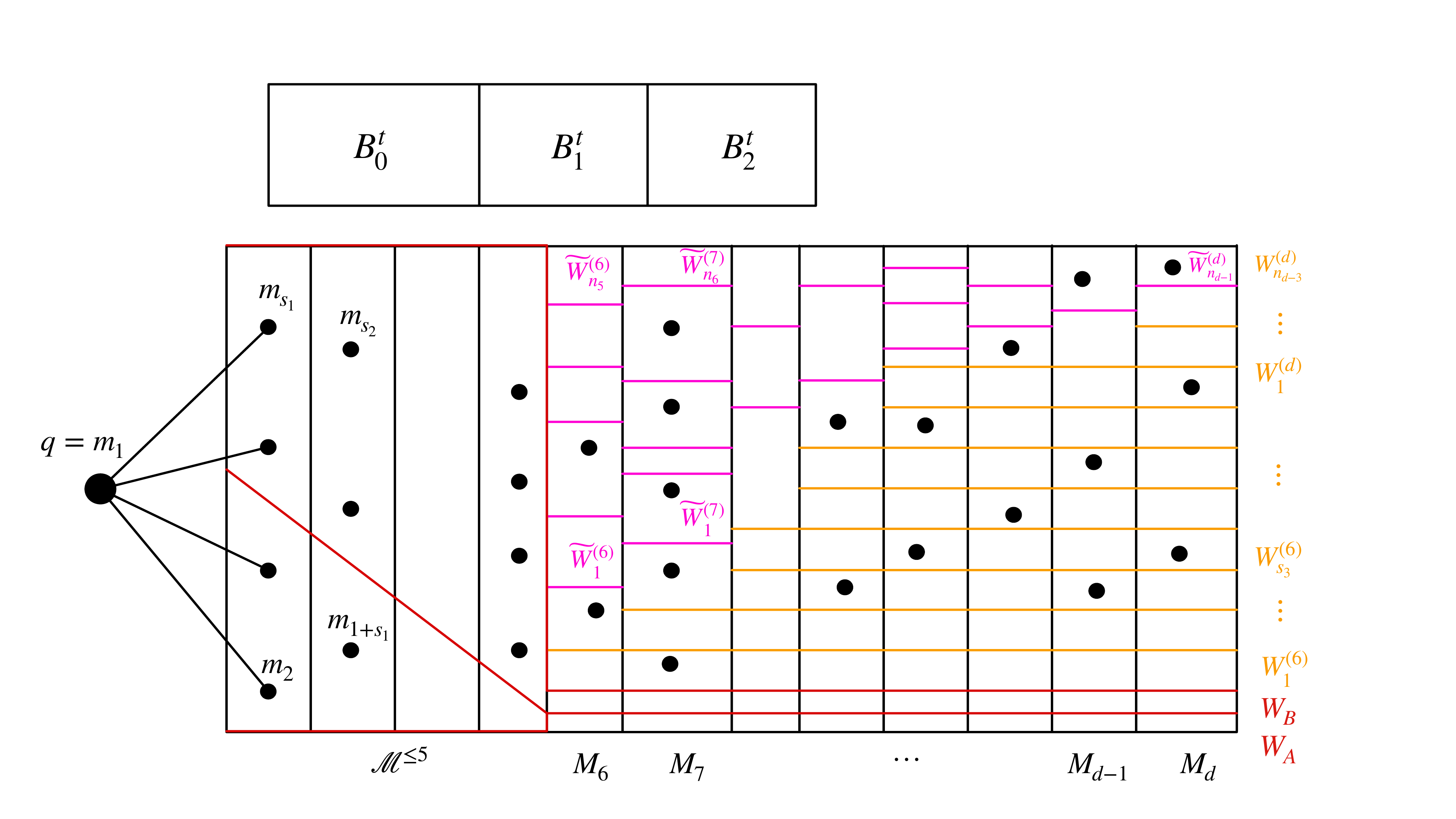}
    \caption{An illustration of the stable grading of $C$ constructed in the proof of \Cref{lemma:contruction_two_paths_optimal}. A vertex is earlier than another if it lies in a part of the stable grading that starts further to the left, or in a part that starts in the same level but lies lower.}
    \label{fig:grading}
\end{figure}
Let \(n_t = \lvert M_t \rvert\) and set \(s_t = \sum_{i = 0}^{t} n_i\). Let \(m_1, \dots, m_{s_{d - 1}}\) be an ordering of the vertices in \(\mathcal{M}^{\leq d - 1}\) such that, whenever \(1 \leq i < j \leq s_{d - 1}\), the \(G[\mathcal{M}^{\leq d -1}]\)-distance between \(q\) and \(m_i\) is at most the \(G[\mathcal{M}^{\leq d - 1}]\)-distance between \(q\) and \(m_j\).

We first construct a stable grading of $C_1 \cup C_2$. We begin with the case \( t = 6 \).
For \(1 \leq i \leq s_{3}\), let \(W_i^{(6)}\) be the set of vertices \(v\) in \(\mathcal{M}^{\geq 6} \setminus N_G^5[q]\) such that \(v\) has a neighbor in \(B_1^6 \cup B_2^6\) that is adjacent to \(m_i\), and has no neighbor in \(B_1^6 \cup B_2^6\) adjacent to any of \(m_1, \dots, m_{i - 1}\).
For \(7 \leq t \leq d\) and \(1 \leq i \leq n_{t - 3}\), let \(W_i^{(t)}\) be the set of vertices \(v\) in \(\mathcal{M}^{\geq t} \setminus N_G^5[q]\) such that \(v\) has a neighbor in \(B_1^t \cup B_2^t\) that is adjacent to \(m_{i + s_{t - 4}} \in M_{t - 3}\) and no neighbor in \(B_1^t \cup B_2^t\) that is adjacent to one of \(m_1, \dots, m_{(i -1 ) + s_{t - 4}}\).
Each set $W_i^{(t)}$ is stable since \(G\) has odd girth at least \(7\). 
Hence, ordered lexicographically by \(t\) and then by \(i\), the sets
\[
\left(W_1^{(6)}, \dots, W_{s_{3}}^{(6)}, W_1^{(7)}, \dots, W_{n_4}^{(7)}, \dots, W^{(d)}_1, \dots, W^{(d)}_{n_{d - 3}}\right)
\]
form a stable grading of \(C_1 \cup C_2\).
We now construct a stable grading of \(C_0\).
For \(6 \leq t \leq d\) and \(1 \leq i \leq n_{t - 1}\),
let \(\widetilde{W}^{(t)}_i\) be the set of vertices in \(C_0^t\) that are adjacent to \(m_{i + s_{t - 2}}\), and are non-adjacent to \(m_1, \dots, m_{(i - 1) + s_{t - 2}}\). 
Since \(G\) is triangle-free, the sets
\[
\left(\widetilde{W}^{(6)}_1, \dots, \widetilde{W}_{n_5}^{(6)}, \dots, \widetilde{W}_{1}^{(d)}, \dots, \widetilde{W}_{n_{d - 1}}^{(d)}\right)
\]
form a stable grading of \( C_0 \).
By combining these stable gradings, we get that
\[
\left(W_A, W_B, W^{(6)}_1, \dots, W_{s_{3}}^{(6)}, \widetilde{W}^{(6)}_1, \dots, \widetilde{W}^{(6)}_{n_5}, \dots,  W_1^{(d)}, \dots, W_{n_{d - 3}}^{(d)}, \widetilde{W}_1^{(d)}, \dots, \widetilde{W}_{n_{d - 1}}^{(d)} \right)
\]
is a stable grading of $C$, where $(W_A, W_B)$ is any bipartition of $C \cap N_G^5[q]$. When convenient, we regard the same sequence, with \( W_A\) and \( W_B\) omitted, as a stable grading of \( C \setminus N^5_G[q]\).

\textbf{Constructing odd- and even-length paths:}
By \Cref{lemma:exclude_bipartite}, we have that $\chi(C \setminus N_G^5[q]) \geq \chi(C) - 1 \geq c + 6$.
From \Cref{lemma2:beyond_leveling} we get that there exist a subset $C'$ of $C \setminus N_G^5[q]$ and an edge $u'q'$ of $G[C_1], G[C_2]$, or $G[C_0^t]$, for some $t \geq 6$, with the following properties:
\begin{itemize}
    \item $G[C' \cup \{q'\}]$ is connected,
    \item $\chi(C') \geq c$,
    \item $u'$ and $q'$ are earlier than every vertex in $C'$, and
    \item $u'$ has no neighbor in $C'$.
\end{itemize}
Observe that the $G$-distance between $q$ and any vertex in $C' \cup \{u', q'\}$ is at least \(6\).
We perform a case distinction according to which of the above induced subgraphs contains the edge \(u'q'\).

\textbf{Case 1:} $u'q'$ is an edge of $G[C_0^t]$ for some $6 \leq t \leq d$.
The situation is illustrated in~\Cref{fig:case1_arithmetic}.
\begin{figure}[ht]
    \centering
    \includegraphics[width=\textwidth, trim={26 0 -26 0}, clip]{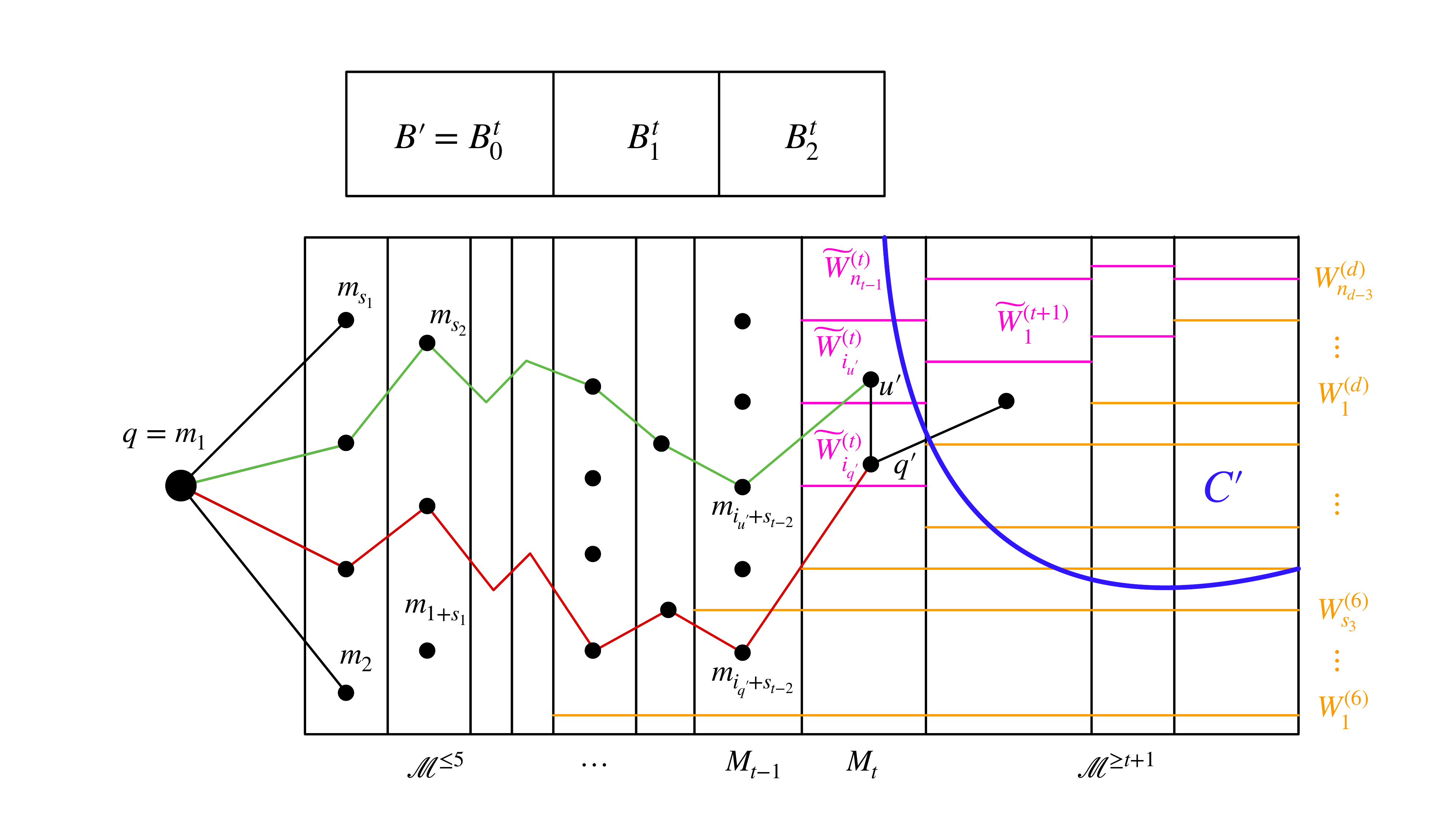}
    \caption{A possible situation in the first case of the proof of \Cref{lemma:contruction_two_paths_optimal}.}
    \label{fig:case1_arithmetic}
\end{figure}
Let $i_{u'}$  and \(i_{q'}\) be such that $u' \in \widetilde{W}^{(t)}_{i_{u'}}$ and $q' \in \widetilde{W}^{(t)}_{i_{q'}}$.
Let $Q_{u'}$ be an induced path in $G[\mathcal{M}^{\leq t - 2} \cup \{m_{i_{u'} + s_{t - 2}}, u'\}]$ between $q$ and $u'$ of length $t$ and let $Q_{q'}$ be an induced path in $G[\mathcal{M}^{\leq t -2} \cup \{m_{i_{q'} + s_{t - 2}}, q'\}]$ between $q$ and $q'$ of length $t$. 
Let \(Q_{u'q'}\) be the path of length \(t + 1\) between $q$ and $q'$ obtained by extending $Q_{u'}$ by the edge $u'q'$. This is an induced path since $G$ is triangle-free and $\mathcal{M}^{\leq t - 2}$ is anticomplete to $M_{t}$.
Thus \(Q_{q'}\) and \(Q_{u'q'}\) are induced paths between \(q\) and \(q'\) of lengths \(t\) and \(t+1\), respectively. Let \(Q_1\) denote the odd-length path and let \(Q_2\) denote the even-length path.

For $B' = B^t_0$, we verify the remaining required properties.
\begin{itemize}
    \item Every vertex in $\mathcal{M}^{\geq t}$ with a neighbor in $B$ that is adjacent to a vertex in $\mathcal{M}^{\leq t - 3}$ is, by construction of the stable grading, earlier than $u'$ and $q'$. Hence no vertex in $C'$ has a neighbor in $B_1^t \cup B_2^t$. Since \(B\) covers \(C\), it follows that $B' = B_0^t$ covers $C'$.
    \item By definition, $B' = B_0^t$ is anticomplete to $\mathcal{M}^{\leq t - 3}$ and  $\left(V(Q_1) \cup V(Q_2) \right) \setminus N^3[q'] \subseteq \mathcal{M}^{\leq t - 3}$.
    \item We have that $V(Q_1) \cup V(Q_2) \subseteq C \cup \{q\}$.
    \item Since $u'$ and $q'$ are earlier than every vertex in $C'$, no vertex in $C'$ is adjacent to any of $m_1, \dots, m_{k}$, where $k = \max \{ i_{u'} + s_{t - 2}, i_{q'} + s_{t - 2} \}$. Since $u'$ does not have a neighbor in $C'$, $C'$ is anticomplete to $\left( V(Q_1) \cup V(Q_2)\right) \setminus \{q'\} \subseteq \{m_1, \dots, m_k\} \cup \{u'\}$.
\end{itemize}

\textbf{Case 2:} $u'q'$ is an edge of $G\left[C_h \right]$ for $h \in \{1, 2\}$.
One possible situation is illustrated in~\Cref{fig:case2_arithmetic}.
\begin{figure}[ht]
    \centering
    \includegraphics[width=\textwidth, trim={-71 0 71 0}, clip]{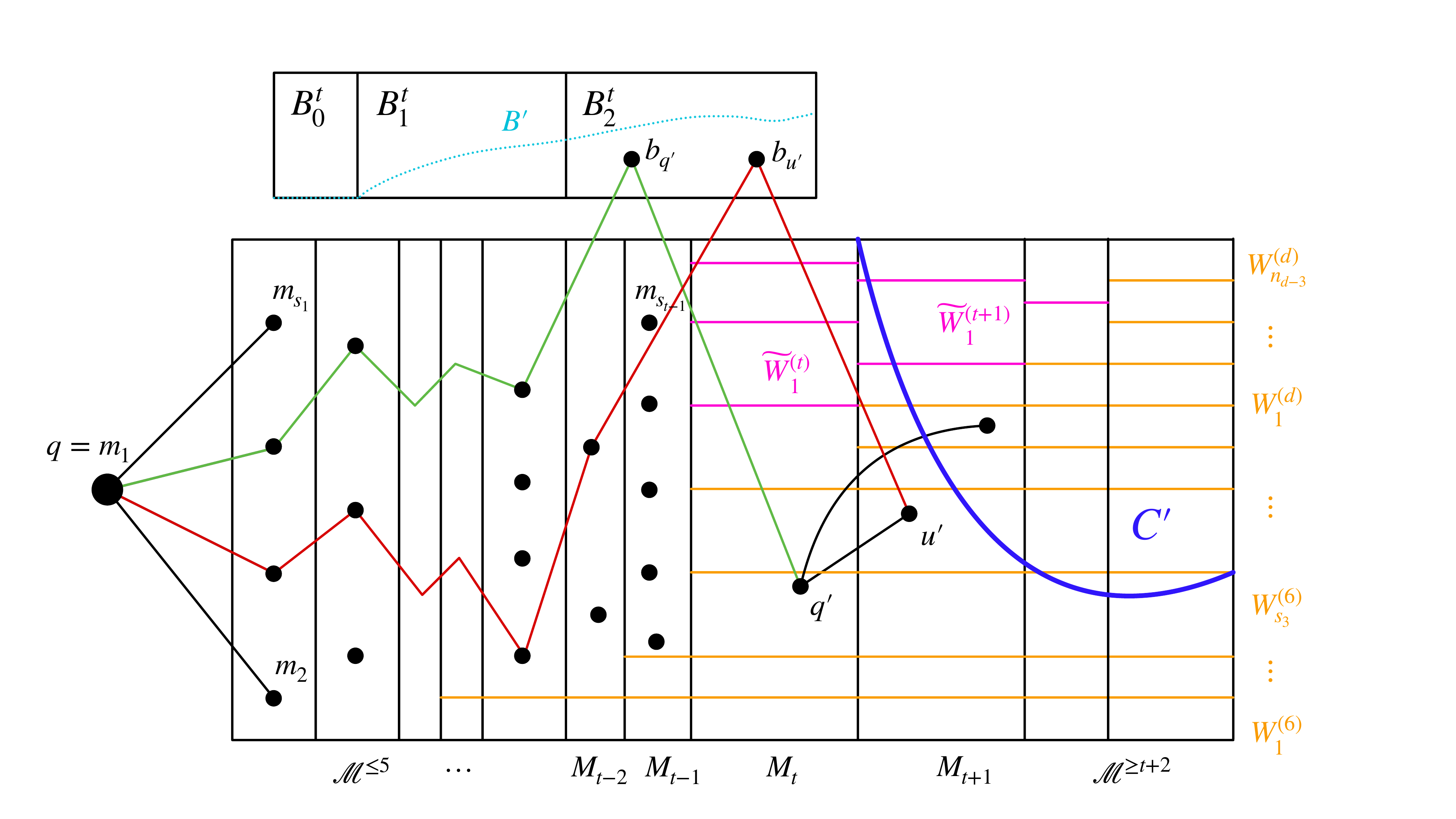}
    \caption{A possible situation in the second case of the proof of \Cref{lemma:contruction_two_paths_optimal}.}
    \label{fig:case2_arithmetic}
\end{figure}
Let $t$ be the largest integer such that $u', q' \in M_t \cup M_{t + 1}$.
For $v' \in \{u', q'\}$, let $b_{v'}$ be a vertex in $B$ adjacent to $v'$ such that $b_{v'}$ is adjacent to $m_{\alpha_{v'}}$ and no other vertex in $B \cap N(v')$ is adjacent to any of $m_1, \dots, m_{\alpha_{v'} - 1}$, where $\alpha_{v'}$ is the minimum index.
This is well-defined since $b_{v'}$ is as in the definition of $C_1$ and $C_2$, where it is additionally chosen to be adjacent to a vertex with minimum index.
By construction, we have that $b_{u'}, b_{q'} \in B^{t + 1}_h \supseteq B^t_h$, where $b_{u'}$ (or $b_{q'}$) is contained in $B^t_h$ if $u' \in M_t$ (or $q' \in M_t$).
Let $t_{u'}$ and $t_{q'}$ be such that $m_{\alpha_{u'}}$ is contained in $M_{t_{u'}}$ and $m_{\alpha_{q'}}$ is contained in $M_{t_{q'}}$.
Moreover, let \(k = \max\{ \alpha_{u'}, \alpha_{q'} \} \) and \(B' = B \setminus N(\{m_1, \dots, m_{k}\})\).

Since $G$ is triangle-free, $b_{u'}$ is non-adjacent to $q'$. Moreover, since the $G$-distance between $q$ and $C' \cup \{u', q'\}$ is at least $6$, neither $b_{u'}$ nor $b_{q'}$ is in $N^4[q]$.
Let $Q_{h}$ be an induced path in $G[\mathcal{M}^{\leq t_{u'}} \cup \{b_{u'}, u', q'\}]$ between $q$ and $q'$ obtained by extending a shortest path from $q$ to $b_{u'}$ in $G[\mathcal{M}^{\leq t_{u'}} \cup \{b_{u'}\}]$ through $m_{\alpha_{u'}}$ by the edges $b_{u'}u'$ and $u'q'$. Since $u'$ is contained in $C_h$, $Q_{h}$ has odd length if $h = 1$ and even length if $h = 2$.
Similarly, let $Q_{3 - h}$ be an induced path in $G[\mathcal{M}^{\leq t_{q'}} \cup \{b_{q'}, q'\}]$ between $q$ and $q'$ obtained by extending a shortest path from $q$ to $b_{q'}$ in $G[\mathcal{M}^{\leq t_{q'}} \cup \{b_{q'}\}]$ through $m_{\alpha_{q'}}$ by the edge $b_{q'}q'$. Since $q'$ is contained in $C_h$, $Q_{3 - h}$ has odd length if $3 - h = 1$ and even length if $3 - h = 2$. Thus \(Q_h\) has parity \(h\), while \(Q_{3-h}\) has parity \(3-h\). In particular, \(Q_1\) has odd length and \(Q_2\) has even length.
Note that $Q_{3 - h}$ and $Q_{h}$ are both induced paths because $\{m_1, \dots, m_k \} \subseteq \mathcal{M}^{\leq (t + 1) - 3} = \mathcal{M}^{\leq t - 2}$. Thus no vertex in $\{m_1, \dots, m_k\}$ is adjacent to a vertex in $\{u', q'\} \subseteq M_t \cup M_{t + 1}$.

Finally, we again verify the remaining required properties.
\begin{itemize}
    \item Since $u'$ and $q'$ are earlier than every vertex in $C'$, no vertex in $C'$ has a neighbor in $B$ that is adjacent to one of $m_1, \dots, m_k$. Hence $B'= B \setminus N(\{m_1, \dots, m_{k}\})$ covers $C'$ because $B$ covers $C$.
    \item By definition, $B'$ is anticomplete to $\{m_1, \dots, m_k\}$ and $\left( V(Q_1) \cup V(Q_2) \right) \setminus N^3[q']$ is contained in $\{ m_1, \dots, m_k\}$. Thus $B'$ is anticomplete to $\left( V(Q_1) \cup V(Q_2) \right) \setminus N^3[q']$.
    \item Since the $G$-distance between $q$ and $\{u', q'\}$ is at least 6, $b_{u'}$ and $b_{q'}$ are not contained in $N^4[q]$. Hence $\left( V(Q_1) \cup V(Q_2) \right)$ is contained in $C \cup \{q\} \cup ((B \cap N^2[q']) \setminus N^4[q])$.
    \item Let $W^{(\ell)}_i$ denote the set in the stable grading containing the later of $u'$ and $q'$. By construction, $\{m_1, \dots, m_k\} \subseteq \mathcal{M}^{\leq \ell - 3}$. Every vertex in $C'$ is later than $u'$ and $q'$, so it is in $\mathcal{M}^{\geq \ell}$ and has no neighbor in $B$ adjacent to one of $m_1, \dots, m_k$. Thus $C'$ is anticomplete to $\{m_1, \dots, m_k\} \cup \{b_{u'}, b_{q'}\}$. Since $u'$ does not have a neighbor in $C'$, $C'$ is anticomplete to $\left( V(Q_1) \cup V(Q_2) \right) \setminus \{q'\} \subseteq \{m_1, \dots, m_k\} \cup \{b_{u'}, b_{q'}, u'\}.$
\end{itemize}
\end{proof}

\begin{remark}
     We require that the odd girth is at least 13 instead of 11 so that we can partition $B$ depending on whether a vertex in $B$ has a neighbor in $\mathcal{M}^{\leq t- 3}$. The authors of~\cite{chudnovsky2024colouringtperfectgraphs} consider the same partition for a fixed level $M_t$ with respect to $\mathcal{M}^{\leq t - 2}$. If we were to use $\mathcal{M}^{\leq t- 2}$ in our proof, one of the paths we find in the second case might not be induced. For example, if $q' \in M_t$ and $u' \in M_{t + 1}$, there could be an edge between $q'$ and $m_{\alpha_{u'}}$, meaning that $Q_{h}$ is not an induced path.
\end{remark}

\section{Constructing segments without even paths}
\label{sec:segments_even}

We say that $w \in V(G)$ is vertex-left-active with respect to a vertex set $X$ if there is a vertex $u \in X$ that is earlier than $w$ and adjacent to it.

\begin{lemma}\label{lemma1:quasi}
    Let $c \geq 1$ and let $G$ be a graph with $\chi(G) \geq c + 2$ and $V(G) = C_0 \cup C_1 \cup C_2$ and let $(W_1, \dots, W_n)$ be a stable grading of $G$.
    Then there exist a subset $X$ of $V(G)$ with $\chi(X) \geq c$ and either a vertex $q$ in $C_0 \cup C_1$ such that 
    \begin{itemize}
        \item $G[X \cup \{q\}]$ is connected, and
        \item $q$ is earlier than every vertex in $X$,
    \end{itemize}
    or an edge $uv$ belonging to $G[C_2]$ such that 
    \begin{itemize}
        \item $G[X]$ is connected,
        \item $u$ and $v$ are earlier than every vertex in $X$, and
        \item at least one of $u$ and $v$ has a neighbor in $X$.
    \end{itemize}
\end{lemma}

\begin{proof}
    Let $A$ be the set of vertices that are vertex-left-active with respect to $C_0 \cup C_1$ or left-active with respect to $C_2$.

    \textbf{Case 1:} $\chi(A) \geq c$.
    Let $D$ be a connected component of $G[A]$ with maximum chromatic number. Then there is a minimum integer $i$ such that $1 \leq i \leq n$ and $W_i \cap V(D)$ is non-empty. Let $w$ be a vertex in $W_i \cap V(D)$. Then there is either a vertex $q \in C_0 \cup C_1$ or an edge $uv$ of $G[C_2]$ for which the required properties are satisfied with $X = V(D)$.

    \textbf{Case 2:} $\chi(A) \leq c - 1$. 
    Let $B_h := C_h \setminus A$ for $h \in \{0, 1, 2\}$ and let $B = B_0 \cup B_1 \cup B_2$. We have that 
    \[
    \chi(B) \geq \chi(G) - \chi(A) \geq c+ 2 - (c - 1) = 3.
    \]
    Suppose that $\chi(B_0 \cup B_1) \geq 2$. Then there is an edge $qq'$ in $G[B_0 \cup B_1]$ with $q \in W_{\ell}$ and $q' \in W_{\ell'}$ for some $1 \leq \ell, \ell' \leq n$ with $\ell \neq \ell'$. However, then either $q$ or $q'$ is vertex-left-active for $C_0 \cup C_1$, which contradicts $q, q' \not \in A$. Hence we may assume that $B_0 \cup B_1$ is a stable set.
    Consider an odd cycle $D$ in $G[B]$, which exists because $\chi(B) \geq 3$.
    Orient each edge from the earlier endpoint to the later endpoint, and denote the resulting orientation by \(\vec{D}\).
    Observe that every vertex in $V(D) \cap (B_0 \cup B_1)$ is a sink in \(\vec{D}\), as otherwise one of its neighbors would be vertex-left-active with respect to $B_0 \cup B_1$. Moreover, since no vertex in $D$ is left-active with respect to $C_2$, there is no directed path of length $2$ in \(\vec{D}\). Indeed, if \(u \to v \to w\) is such a path, then $uv$ is an edge of $G[C_2]$ since neither $u$ nor $v$ is a sink, both $u$ and $v$ are earlier than $w$, and $w$ is adjacent to $v$, contradicting $w \in B$. Hence every vertex of $\vec{D}$ is either a source or a sink. Thus the sources and sinks form a bipartition of $D$, contradicting that $D$ is an odd cycle.
\end{proof}

\begin{lemma}\label{lemma2:quasi}
    Let $c \geq 1$ and let $G$ be a triangle-free graph with $\chi(G) \geq c + 3$ and $V(G) = C_0 \cup C_1 \cup C_2$ and let $(W_1, \dots, W_n)$ be a stable grading of $G$.

    Then there exist a subset $X$ of $V(G)$ with $\chi(X) \geq c$ and either a vertex $q$ in $C_0 \cup C_1$ such that 
    \begin{itemize}
        \item $G[X \cup \{q\}]$ is connected, and
        \item $q$ is earlier than every vertex in $X$,
    \end{itemize}
    or an edge $uv$ in $G[C_2]$ such that 
    \begin{itemize}
        \item $G[X]$ is connected,
        \item $u$ and $v$ are earlier than every vertex in $X$,
        \item $u$ has no neighbor in $X$, and
        \item $v$ has a neighbor in $X$.
    \end{itemize}
\end{lemma}
The proof of \Cref{lemma2:quasi} is completely analogous to the proof of \Cref{lemma2:beyond_leveling}, which is why we refrain from repeating it.

We are now ready to prove \Cref{lemma:construction_quasi}.

\begin{proof}[Proof of \Cref{lemma:construction_quasi}]
    For $i \geq 0$, let $M_i$ be the vertices in $C \cup \{q\}$ with $G\left[ C \cup \{q\}\right]$-distance exactly $i$ from $q$. Let $d$ be the unique integer such that $M_d \neq \emptyset$ and $M_{d'} = \emptyset$ for all $d' > d$.
    For $1 \leq t \leq d$, let $\mathcal{M}^{\leq t} = M_0 \cup \dots \cup M_t$ and $\mathcal{M}^{\geq t} = M_t \cup \dots \cup M_d$.
    For $6 \leq t \leq d$, let $B_0^t, B_1^t, B_2^t$, $C_0, C_1, C_2, C_0^t, C_1^t, C_2^t$, and the stable grading
    \begin{align*}
        \left(W^{(6)}_1, \dots, W_{s_{3}}^{(6)}, \widetilde{W}^{(6)}_1, \dots, \widetilde{W}^{(6)}_{n_5}, \dots,  W_1^{(d)}, \dots, W_{n_{d - 3}}^{(d)}, \widetilde{W}_1^{(d)}, \dots, \widetilde{W}_{n_{d - 1}}^{(d)} \right)
    \end{align*}
    of $C \setminus N^5_G[q]$ be as in the proof of \Cref{lemma:contruction_two_paths_optimal}.
    Moreover, by excluding $N^6(q)$ from each part of the stable grading, we get a stable grading of $C \setminus  N_G^6[q]$.
    By \Cref{lemma:exclude_bipartite}, we have that $\chi(C \setminus N_G^6[q]) \geq \chi(C) - 1 \geq c + 3$.
    By \Cref{lemma2:quasi}, there is a subset $\widehat{C}$ of $C \setminus N_G^6[q]$ with $\chi(\widehat{C}) \geq c$ and there is either a vertex $\hat{q}$ in $C_0 \cup C_2$ such that 
    \begin{itemize}
        \item $G[\widehat{C} \cup \{\hat{q}\}]$ is connected, and
        \item ${\hat{q}}$ is earlier than every vertex in $\widehat{C}$,
    \end{itemize}
    or an edge $u'q'$ belonging to $G[C_1]$ such that 
    \begin{itemize}
        \item $G[\widehat{C} \cup \{q'\}]$ is connected,
        \item $u'$ and $q'$ are earlier than every vertex in $\widehat{C}$, and 
        \item $u'$ has no neighbor in $\widehat{C}$.
    \end{itemize}
Given such an edge $u'q'$ in $G[C_1]$, we can find an odd- and even-length induced path with the desired properties for $C' = \widehat{C}$ as in the second case in the proof of \Cref{lemma:contruction_two_paths_optimal}.
Hence we may assume that there is a vertex $\hat{q}$ in $C_0 \cup C_2$ with the above properties. If $\hat{q} \in C_2$ or $\hat{q} \in C_0 \cap M_t$ for an odd $t$, we can construct an induced path of odd length as in first or second case of the proof of \Cref{lemma:contruction_two_paths_optimal}. 

Therefore it suffices to consider the case where $\hat{q} \in C_0 \cap M_t$ for an even $t$ with $7 \leq t \leq d$. Let $i_{\hat{q}}$ be such that $\hat{q} \in \widetilde{W}^{(t)}_{i_{\hat{q}}}$ and let $\hat{u} = m_{i_{\hat{q}} + s_{t - 2}}$. 
Note that the $G$-distance between $q$ and $\widehat{C} \cup \{\hat{q}\}$ is at least 7, so $\hat{u} \in M_{t - 1}$ has $G$-distance at least 6 to $q$. In particular, $\hat{u}$ is earlier than $\hat{q}$ since $\hat{q} \in C_0$. Thus $\hat{u}$ is also earlier than every vertex in $\widehat{C}$.

First we consider the case where $\hat{u} \in C_0$. Let $i_{\hat{u}}$ be such that $\hat{u} \in \widetilde{W}^{(t - 1)}_{i_{\hat{u}}}$. 
Let $Q_1$ be an induced path in $G[\mathcal{M}^{\leq t - 3} \cup \{ m_{i_{\hat{u}} + s_{t - 3}}, \hat{u} \}]$ between $q$ and $\hat{u}$. Since $t$ is even and $\hat{u}\in M_{t - 1}$, $Q_1$ has odd length. Let $q' = \hat{u}$, $C' = \widehat{C} \cup \{\hat{q} \}$, and $B' = B_0^{t}$. We verify the required properties.
\begin{itemize}
    \item Since $\hat{q}$ is earlier than every vertex in $\widehat{C}$ and $\hat{q} \in M_t \cap C_0$, we have that $C' \subseteq \mathcal{M}^{\geq t}$. Every vertex in $\mathcal{M}^{\geq t}$ with a neighbor in $B$ that is adjacent to a vertex in $\mathcal{M}^{\leq t- 3}$ is earlier than $\hat{q}$ by construction of the stable grading. Hence no vertex in $C'$ has a neighbor in $B^t_1 \cup B^t_2$, and $B' = B^t_0$ covers $C'$ since $B$ covers $C$.
    \item By definition, $B' = B^t_0$ is anticomplete to $\mathcal{M}^{\leq t - 3}$ and is thus also anticomplete to $V(Q_1) \setminus N^3[q'] \subseteq \mathcal{M}^{\leq t - 3}$. 
    \item We have that $V(Q_1) \subseteq C \cup \{q\}$.
    \item Since $\hat{q}$ is earlier than every other vertex in $C'$, no vertex in $C'$ is adjacent to any of $m_1, \dots, m_{i_{\hat{q}} -1 + s_{t - 2}}$. Hence $C'$ is anticomplete to $V(Q_1) \setminus \{q'\} = V(Q_1) \setminus \{m_{i_{\hat{q}} + s_{t - 2}}\}$.
\end{itemize}

Now consider the case where $\hat{u} \in C_h$ for $h \in \{1, 2\}$.
Let $b_{\hat{u}}$ be a vertex in $B$ adjacent to $\hat{u}$ such that $b_{\hat{u}}$ is adjacent to $m_{\alpha_{\hat{u}}}$ and no other vertex in $B \cap N(\hat{u})$ is adjacent to any of $m_1, \dots, m_{\alpha_{\hat{u}} - 1}$, where $\alpha_{\hat{u}}$ is the minimum index.
By construction, we have that $b_{\hat{u}} \in B^t_h$. Let $t_{\hat{u}}$ be such that $m_{\alpha_{\hat{u}}}$ is contained in $M_{t_{\hat{u}}}$ and let $B' = B \setminus N(\{ m_1, \dots, m_{\alpha_{\hat{u}}} \})$. 

We first consider $h = 1$.
Let $Q_1$ be an induced path between $q$ and $\hat{q}$ in $G[\mathcal{M}^{\leq t_{\hat{u}}} \cup \{b_{\hat{u}}, \hat{u}, \hat{q}\}]$ obtained by extending a shortest path from $q$ to $b_{\hat{u}}$ in $G[ \mathcal{M}^{\leq t_{\hat{u}}} \cup \{b_{\hat{u}}\}]$ by the edges $b_{\hat{u}}\hat{u}$ and $\hat{u}\hat{q}$. Since $\hat{u} \in C_1$, this path has odd length. Let $q' = \hat{q}$ and $C' = \widehat{C}$. We verify the required properties.

\begin{itemize}
    \item No vertex in $C'$ has a neighbor in $B$ that is adjacent to one of $\{m_1, \dots, m_{\alpha_{\hat{u}}}\}$ since $\hat{u}$ is earlier than every vertex in $C'$. Thus $B'$ covers $C'$ because $B$ covers $C$.
    \item Since $V(Q_1) \setminus N^3[q']$ is contained in $\{m_1, \dots, m_{\alpha_{\hat{u}}}\}$ and $B'$ is anticomplete to $\{ m_1, \dots, m_{\alpha_{\hat{u}}} \}$, $B'$ is anticomplete to $V(Q_1) \setminus N^3[q']$.
    \item Since the $G$-distance between $q$ and $q' = \hat{q}$ is at least 7, $b_{\hat{u}}$ has $G$-distance at least 5 to $q$. Hence $V(Q_1)$ is contained in $C \cup \{q\} \cup (( B \cap N^2[q'] ) \setminus N^4[q])$ because $V(Q_1) \cap B = \{b_{\hat{u}}\}$.
    \item Since $q'$ is earlier than every vertex in $C'$, $\hat{u} = m_{i_{\hat{q}} + s_{t -2}}$ does not have a neighbor in $C'$. Since $\hat{u}$ is earlier than every vertex in $C'$, no vertex in $C'$ is adjacent to $b_{\hat{u}}$. Because $C' \subseteq \mathcal{M}^{\geq t}$ is anticomplete to $\mathcal{M}^{\leq t- 2}$, we have that $C'$ is anticomplete to $V(Q_1) \setminus \{q'\}$.
\end{itemize}

Finally, consider $h = 2$. Let $Q_1$ be an induced path in $G[\mathcal{M}^{\leq t_{\hat{u}}} \cup \{b_{\hat{u}}, \hat{u}\}]$ between $q$ and $\hat{u}$ obtained by extending a shortest path between $q$ and $b_{\hat{u}}$ in $G[\mathcal{M}^{\leq t_{\hat{u}}} \cup \{b_{\hat{u}}\}]$ by the edge $b_{\hat{u}}\hat{u}$. Since $\hat{u} \in C_2$, $Q_1$ has odd length. Let $q' = \hat{u}$ and let $C' = \widehat{C} \cup \{ \hat{q} \}$. The desired properties follow analogously as for the case $h = 1$.

\begin{itemize}
    \item No vertex in $C'$ has a neighbor in $B$ that is adjacent to one of $\{m_1, \dots, m_{\alpha_{\hat{u}}}\}$ since $\hat{u}$ is earlier than every vertex in $C'$. Thus $B'$ covers $C'$ because $B$ covers $C$.
    \item Since $V(Q_1) \setminus N^3[q']$ is contained in $\{m_1, \dots, m_{\alpha_{\hat{u}}}\}$ and $B'$ is anticomplete to $\{ m_1, \dots, m_{\alpha_{\hat{u}}} \}$, $B'$ is anticomplete to $V(Q_1) \setminus N^3[q']$.
    \item Since the $G$-distance between $q$ and $q' = \hat{u}$ is at least 6, $b_{\hat{u}}$ has $G$-distance at least 5 to $q$. Hence $V(Q_1)$ is contained in $C \cup \{q\} \cup (( B \cap N^2[q'] ) \setminus N^4[q])$ because $V(Q_1) \cap B = \{b_{\hat{u}}\}$.
    \item Since $q' = \hat{u}$ is earlier than every vertex in $C'$, no vertex in $C'$ is adjacent to $b_{\hat{u}}$. Because $C' \subseteq \mathcal{M}^{\geq t}$ is anticomplete to $\mathcal{M}^{\leq t - 2}$, we have that $C'$ is anticomplete to $V(Q_1) \setminus \{q'\}$.
\end{itemize}
\end{proof}

\section*{Declaration of AI Usage}
An earlier version of this paper, in which we proved that t-perfect graphs are 186-colorable, was written entirely without the use of generative AI. Subsequently, we used ChatGPT-5.6 Sol Pro to explore alternative proof strategies for \Cref{lemma1:beyond_leveling}. This led to a proof of the lemma under a weaker assumption on the chromatic number. Together with other improvements, this contributed to the final bound of 130.

\section*{Acknowledgments}
We thank Sang-il Oum for pointing out a minor mistake in the first version of this paper.

\bibliographystyle{plain}
\bibliography{sources}

\end{document}